\documentclass{article} 

\usepackage{amssymb}
\usepackage{amsfonts}
\usepackage{mathtools}
\usepackage{nicefrac}

\def\oo{\infty}

\def\bbP{\mathbb{P}}

\def\bbR{\mathbb{R}}

\def\boldOne{\boldsymbol{1}}

\makeatletter
\let\IfStar\@ifstar 

\newcommand*{\eqdef}{\@ifstar{\eqdef@B}{\eqdef@A}}
\newcommand*{\eqdef@A}{\stackrel{{\scriptscriptstyle \mathrm{def}}}{\coloneqq}}
\newcommand*{\eqdef@B}{\coloneqq}

\newcommand*{\defeq}{\@ifstar{\defeq@B}{\defeq@A}}
\newcommand*{\defeq@A}{\stackrel{{\scriptscriptstyle \mathrm{def}}}{\eqqcolon}}
\newcommand*{\defeq@B}{\eqqcolon}

\def\@given@A{{\mkern2mu\mid\mkern1.5mu}}
\def\@given@B{\middle|}
\let\given\@given@A
\newcommand*{\useGiven@B}[1]{\def\given{\@given@B}#1\def\given{\@given@A}}

\mathchardef\equals=\mathcode`=
\def\equals@A{\equals}
\def\equals@B{{\mkern0.5mu\equals\mkern1.5mu}}
\def\useEquals@A#1{\begingroup\lccode`~=`=\lowercase{\endgroup\def~}{\equals@A}\mathcode`=="8000{#1}\begingroup\lccode`~=`=\lowercase{\endgroup\def~}{\equals}}
\def\useEquals@B#1{\begingroup\lccode`~=`=\lowercase{\endgroup\def~}{\equals@B}\mathcode`=="8000{#1}\begingroup\lccode`~=`=\lowercase{\endgroup\def~}{\equals}}

\newcommand*{\style@A}[1]{\useEquals@A{#1}}
\newcommand*{\style@B}[1]{\useEquals@B{\useGiven@B{#1}}}

\newcommand*{\set}{\@ifstar{\set@B}{\set@A}}
\newcommand*{\set@A}[1]{\style@A{\{#1\}}}
\newcommand*{\set@B}[1]{\style@B{\left\{#1\right\}}}

\newcommand*{\abs}{\@ifstar{\abs@B}{\abs@A}}
\newcommand*{\abs@A}[1]{\style@A{|#1|}}
\newcommand*{\abs@B}[1]{\style@B{\left|#1\right|}}

\newcommand*{\norm}{\@ifstar{\norm@B}{\norm@A}}
\newcommand*{\norm@A}[1]{\style@A{\Vert#1\Vert}}
\newcommand*{\norm@B}[1]{\style@B{\left\lVert\{#1\right\rVert}}

\newcommand*{\lr}{\@ifstar{\lr@B}{\lr@A}}
\newcommand*{\lr@A}[1]{\style@A{(#1)}}
\newcommand*{\lr@B}[1]{\style@B{\left(#1\right)}}

\newcommand*{\LR}{\@ifstar{\LR@B}{\LR@A}}
\newcommand*{\LR@A}[1]{\style@A{[#1]}}
\newcommand*{\LR@B}[1]{\style@B{\left[#1\right]}}

\newcommand*{\vecList}{\@ifstar{\vecList@B}{\vecList@A}}
\newcommand*{\vecList@A}[1]{{\LR{#1}}^\top}
\newcommand*{\vecList@B}[1]{{\LR*{#1}}^\top}

\makeatother

\newcommand*{\one}[1]{\boldOne\set{#1}\,}

\usepackage{xstring}
\newcommand*{\prob}[1]{\IfStrEq{#1}{_}{\ProbWithSub}{\bbP\lr{#1}}}
\newcommand*{\ProbWithSub}[2]{{\operatorname*{\bbP}_{#1}(#2)}}

\usepackage{makecell}
\usepackage{booktabs}
\usepackage{multirow}
\usepackage{wrapfig}
\usepackage{float}
\usepackage{xcolor}
\usepackage{hyperref}
\definecolor{customBlue}{HTML}{0064bd}
\definecolor{customRed}{HTML}{9F393D}
\hypersetup{
  colorlinks,
  linktoc=all,
  citecolor=green!40!black,
  linkcolor=customRed,
  urlcolor=customRed
}
\usepackage{listings}\definecolor{codebg}{RGB}{248,248,248}
\definecolor{keyword}{RGB}{0,0,180}
\definecolor{comment}{RGB}{0,150,0}
\definecolor{string}{RGB}{180,0,0}
\definecolor{linegray}{RGB}{200,200,200}

\lstdefinestyle{mycode}{
    language=Python,
    backgroundcolor=\color{codebg},   
    basicstyle=\ttfamily\footnotesize,       
    keywordstyle=\color{keyword}\bfseries,
    commentstyle=\color{comment}\itshape,
    stringstyle=\color{string},
    numbers=left,                     
    numberstyle=\tiny\color{linegray},
    stepnumber=1,                      
    numbersep=8pt,
    showstringspaces=false,            
    tabsize=4,
    breaklines=true,                   
    breakatwhitespace=true,
    frame=single,                      
    rulecolor=\color{linegray},
    captionpos=b,                      
    linewidth=\linewidth,              
    xleftmargin=0pt,
    xrightmargin=0pt,
    aboveskip=1em,
    belowskip=1em,
    columns=fullflexible               
}

\usepackage{xspace}
\usepackage{xparse}
\usepackage{parskip}
\usepackage{natbib}

\usepackage{amsthm}
\newtheorem{theorem}{Theorem}[section]

\theoremstyle{definition}

\theoremstyle{remark}

\newtheorem{proposition}[theorem]{Proposition}

\makeatletter\def\thm@space@setup{\thm@preskip=\parskip \thm@postskip=0pt}\makeatother
\let\oldproof\proof\def\proof{\oldproof\unskip}

\newcounter{equationNumber}
\def\tagEquation{\tag{\arabic{equationNumber}}\stepcounter{equationNumber}}

\begin{document}

\title{Uniform Distributions on p-Balls and the Singular Role of p=1,2,\texorpdfstring{$\infty$}{infinity} in p-Norm Geometry}



\author{Carlos Pinzón}


\maketitle

\begin{abstract}
  This paper studies the relationship between volume and surface uniform measures on n-dimensional $p$-balls under the $p$-norm.
  It is proved that for $p=1$, $p=2$ and $p=\infty$, and only for these values of $p$, radial projection maps a volumetrically uniform distribution to a surface-uniform distribution.
  Algorithms for exact and approximate uniform sampling on $p$-balls and $p$-spheres are provided, together with empirical illustrations.
\end{abstract}

\def\Sph{\textsf{Sph}\xspace}
\def\Simplex{\Delta\xspace}
\def\Ball{\textsf{Ball}\xspace}
\def\Leb{{\mu_\text{Leb}}\xspace}
\def\eps{\epsilon}

\def\IndIso{\textnormal{\textsf{IndIso}}\xspace}
\def\NTail{\text{NTail}\xspace}
\def\HTail{\text{HTail}\xspace}

\def\bfx{\mathbf{x}}
\def\bfy{\mathbf{y}}
\def\bf#1{\textbf{#1}}

\def\UniDir{\textnormal{\textsf{UniDir}}\xspace}
\def\Dist{\textnormal{\textsf{Dist}}\xspace}
\def\Rad{\textnormal{\textsf{Rad}}\xspace}
\def\Exp{\textnormal{\textsf{Exp}}\xspace}
\def\Chi{\textnormal{\textsf{Chi}}\xspace}
\let\G\Gamma
\def\Gamma{\textnormal{\textsf{Gamma}}\xspace}
\def\Beta{\textnormal{\textsf{Beta}}\xspace}
\def\Uni{\textnormal{\textsf{Uni}}\xspace}
\def\Lomax{\textnormal{\textsf{Lomax}}\xspace}
\def\Gauss{\textnormal{\textsf{Gauss}}\xspace}
\def\Student{\textnormal{\textsf{Student}}\xspace}
\def\Cauchy{\textnormal{\textsf{Cauchy}}\xspace}
\def\GenGamma{\textnormal{\textsf{GenGamma}}\xspace}
\def\Erlang{\textnormal{\textsf{Erlang}}\xspace}
\def\Weibull{\textnormal{\textsf{Weibull}}\xspace}
\def\InvGamma{\textnormal{\textsf{InvGamma}}\xspace}
\def\Levy{\textnormal{\textsf{Levy}}\xspace}
\def\UniParetoMix{\textnormal{\textsf{UniParetoMix}}\xspace}
\NewDocumentCommand{\nD}{O{n}}{^{(#1)}}
\NewDocumentCommand{\n}{O{n}}{_{1..#1}}
\NewDocumentCommand{\DKL}{s O{} O{}}{
    \IfBooleanTF{#1}
        {{D_{\text{KL}}}}
        {{\DKL*(#2\,||\,#3)}}
}
\NewDocumentCommand{\Lpq}{O{q}}{{\ell_{p,#1}}}

\section{Introduction}

The \emph{$p$-norms} are a family of functions that measure the magnitude of $n$-dimensional vectors in $\bbR^n$ as\[
  \norm{\bfx}_p \eqdef \begin{cases}
    \lr*{\sum_{i=1}^n |x_i|^p}^{1/p} & \text{if } p\in(0,\oo) \\
    \lim_{p\to\oo} \norm{\bfx}_p = \max_{i=1}^n |x_i| & \text{if } p=\oo.
  \end{cases}
\]

These norms are ubiquitous across mathematics, statistics, and machine learning.
For $p\geq 1$, the induced \emph{$p$-distance} function $\norm{x-y}_p$ defines a metric for $\bbR^n$, known as the Minkowski distance, which forms the basis for several important functional spaces, including $\ell^p$ and $L^p$ spaces.


The norms $p=1$, $p=2$, and $p=\oo$ are fundamentally special.
Geometrically, their unit $p$-balls yield simple shapes: diamonds, circles, and squares in two dimensions.
Their induced distances\em known as the Manhattan, Euclidean, and Chebyshev metrics\em are canonical examples of metrics in the theory of metric spaces.

Furthermore, in statistics, the $p$-norms can be used to define the most fundamental central tendency statistics as $m_p(X) \eqdef \inf_m E(\norm{X-m}_p)$.
Notably, the mean occurs with $p=2$, the median with $p=1$, the midrange with $p=\oo$.
Three important limits also occur as $p\to 0^+$, namely, the geometric mean via $\norm{x}_p (1/n)^{1/p} \to (\prod_{i=1}^n x_i)^{1/n}$, the counting measure via $\norm{x}_p^p \to \sum_{i=1}^n \one{x_i \neq 0}$ (this is the one related to the mode), and the plain limit $\norm{x}_p \to \oo$ for all $x\neq 0$.

Lastly, in machine learning, the $p$-norms are used as regularization penalties to reduce overfitting ---mainly L1 and L2 penalties, and as loss functions in regression problems.


This paper identifies a deeper structural property: for these three norms, and only these, radial projection from a volumetrically uniform distribution on the $n$-dimensional $p$-ball yields a superficially uniform distribution on the $p$-sphere.
This equivalence fails for all other $p$, highlighting a structural discontinuity in the geometry of $p$-norms.

\section{Squigonometric functions}
\def\arc{\mathrm{arc}}
\def\sinarc{\sin^\arc}
\def\cosarc{\cos^\arc}
\def\piarc{\pi^\arc}

\begin{figure}[ht]
  \hspace{-0.1em}
  \begin{minipage}{\textwidth}
    \includegraphics[width=0.50\textwidth]{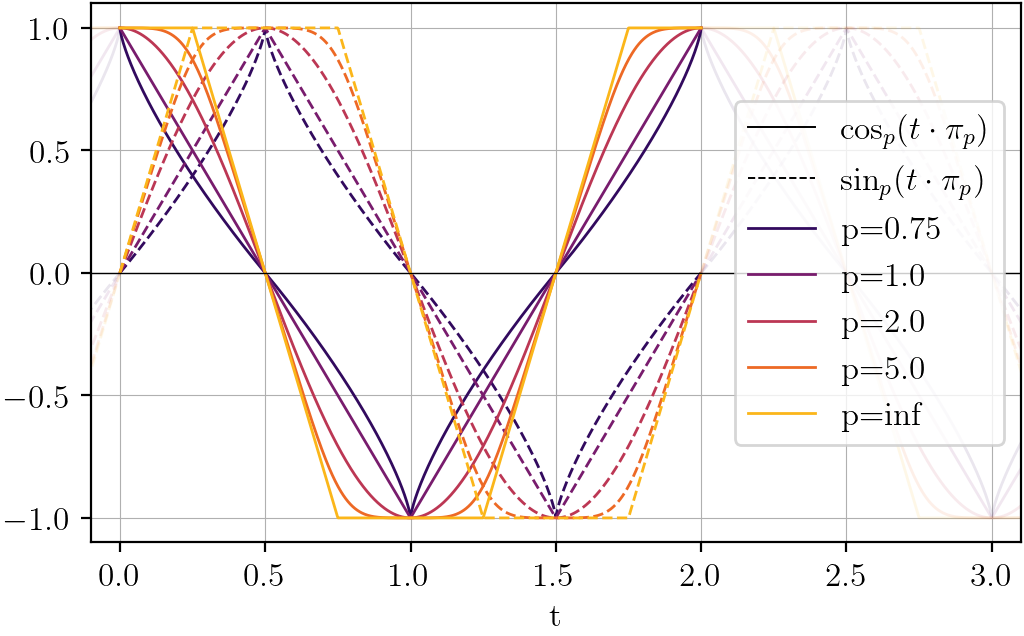}
    \includegraphics[width=0.475\textwidth]{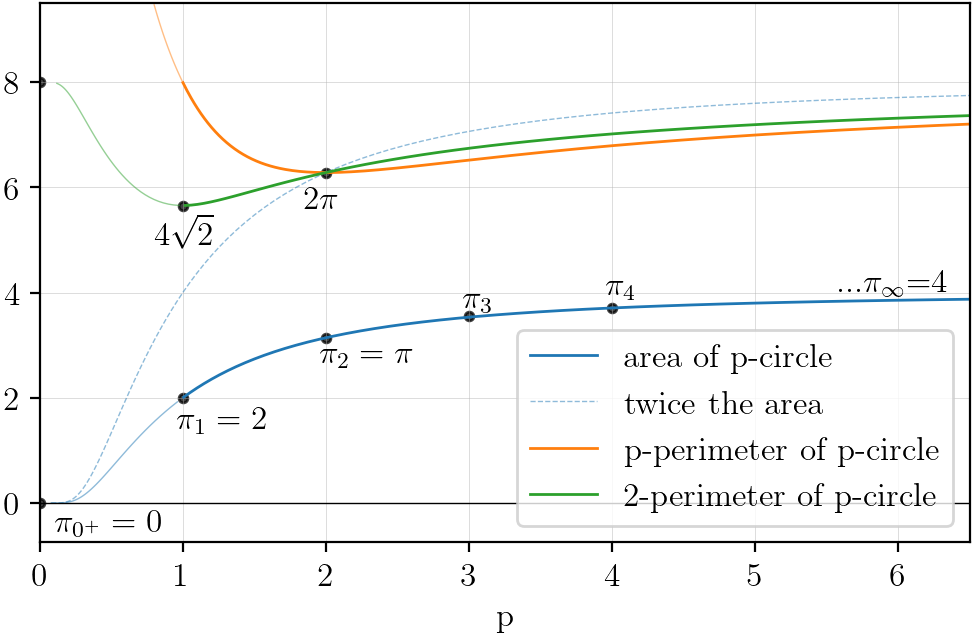}
  \end{minipage}
  \caption{Squigonometric functions (left) and areas and perimeters of the unit $p$-circle in different $p$-norms (right).}
  \label{fig:p-perimeters}
\end{figure}

To parametrize the two dimensional unit $p$-sphere and $p$-ball, also known as the $p$-circle, we employ the \emph{squigonometric functions} introduced in~\cite{wood2020squigonometry-chaps3-7}.
These term squigonometric reflects the shape of the $p$-circle, which lies between a square ($p=\oo$) and a circle ($p=2$)—a form often called a ``squircle`' in~\cite{wood2020squigonometry-chaps3-7}.

Let $\sin_p(t)$ and $\cos_p(t)$ denote the generalized sine and cosine functions.
For $p<\oo$, there are defined (Section 3.1 of \cite{wood2020squigonometry-chaps3-7}) as the unique solution to the coupled initial value problem\[
  C'(t) = -S(t)^{p-1} ;\quad
  S'(t) = C(t)^{p-1} ;\quad
  C(0) = 1 ;\quad
  S(0) = 0.
  \tagEquation\label{eq:squigonometric-ode}
  \]
In the limiting case $p\to\oo$, the functions converge to piecewise linear analogues $\sin_\oo(t) \eqdef \cos_\oo(t-2)$ and \[
\cos_\oo(t) \eqdef \begin{cases}
\min(1, \max(-1, |t-4|-2)) & \text{if } t\in[0, 8] \\
\cos_\oo(t \mod 8) & \text{otherwise}. \\
\end{cases}
\]

The curve $(\cos_p(t),\sin_p(t))$ as $t$ varies from $0$ to $2\pi_p$, where $\pi_p$ is the area of the $p$-circle, traces the boundary of the unit $p$-circle, as shown in Figure~\ref{fig:2D} for $t\in[0,\pi_p/2]$.
Furthermore, $\cos_p(t)$ and $\sin_p(t)$ satisfy the following properties:
\begin{enumerate}
  \item $\pi_p$ is the area of the $p$-circle.
  \item $\norm{(\cos_p(t),\sin_p(t))}_p = 1$ for all $t\in\bbR$.
  \item $\cos_p(t) = \cos_p(-t) = \sin_p(t-\pi_p/2)$ for all $t\in\bbR$.
\end{enumerate}

These generalized trigonometric functions, depicted in Figure~\ref{fig:p-perimeters}, provide an area-based parametrization of the $p$-circle that plays a central role in the geometric analysis that follows. 



\section{Area vs. perimeter in 2D}

We now study the relationship between area and arc length on the $p$-circle.
Consider the first quadrant, parametrized by $\set{(\cos_p(t), \sin_p(t)): t\in[0,\pi_p/2]}$.
In reference to Figure~\ref{fig:2D}, define $A(t)$ as the area of the orange shade, which is enclosed by the circumference of the $p$-circle to the right, by the x-axis to the bottom, and by the line that crosses $(0,0)$ and $(\cos_p(t), \sin_p(t))$ to the top, and let $\Lpq(t)$ be the \emph{$q$-length} of the black curve $(1,0)$ to $(\cos_p(t), \sin_p(t))$.
The parameter $q$ defines how the length is measured, and the only relevant values are $q=2$ for euclidean length and $q=p$ for the metric induced by $p$ itself.

\begin{figure}[ht]
  \centering
  \raisebox{0em}{
    \includegraphics[width=0.5\textwidth]{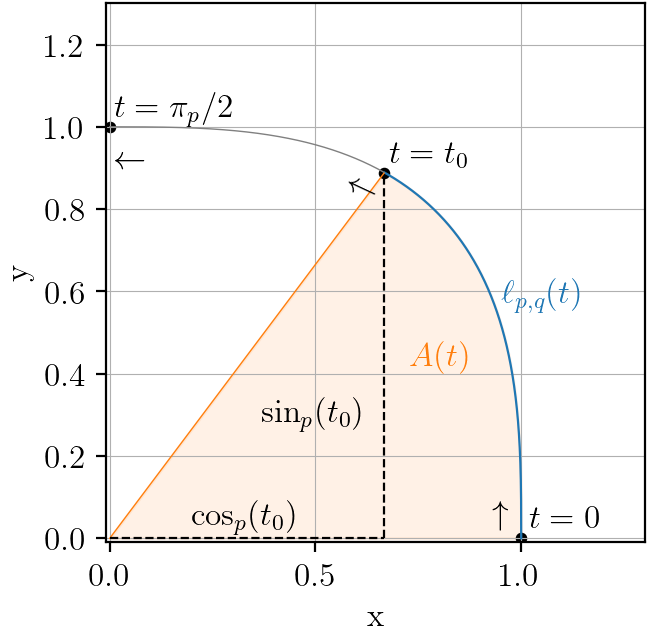}
  }
  \caption{Parametrization of a quarter of the $(n{=}2)$-dimensional $(p{=}4)$-ball showing cumulative area (orange) and $q$-length (blue).}
  \label{fig:2D}
\end{figure}

The $q$-length, can be computed as $\Lpq(t_0) = \int_0^{t_0} d\,\Lpq(t)$, where the infinitesimal $d\,\Lpq(t)$ is the $q$-norm of a triangle with sides $d\,\cos_p(t)$ and $d\,\sin_p(t)$.
Following the ODE in Equation~\eqref{eq:squigonometric-ode} for $p<\oo$, the two sides become $-\sin_p(t)^{p-1}\,dt$ and $\cos_p(t)^{p-1}\,dt$ respectively, and for $p=\oo$, the triangle degenerates into a single line of length $dt$ (regardless of $q$).
Hence, \[
  \frac{d\,\Lpq(t)}{dt} = \begin{cases}
    \norm{(\cos_p(t)^{p-1}, \sin_p(t)^{p-1})}_q &\text{ if } p<\oo\\
    1 &\text{ if }  p=\oo. \\
  \end{cases}
  \tagEquation\label{eq:dLqdt}
\]

It is known that $A(t)$ is proportional to $t$ (more precisely $A(t) = \frac{t}{2}$), which is why this parametrization is called area-based~\cite{wood2020squigonometry-chaps3-7}, pages 59-60. Indeed, since
\begin{align*}
  A(t_0) =& \int_{\pi_p/2}^{t_0} \cos_p(t) \frac{\sin_p(t_0)}{\cos_p(t_0)}\;d\,\cos_p(t) + \int_{t_0}^0 \sin_p(t)\;d\,\cos_p(t)
  \\ =& \frac{\sin_p(t_0)}{\cos_p(t_0)} \left.\frac{\cos_p(t)^2}{2}\right|_{\pi_p/2}^{t_0} + \int_{t_0}^0 \sin_p(t)\;d\,\cos_p(t)
  \\ =& \frac{1}{2} \sin_p(t_0) \cos_p(t_0) + \int_{t_0}^0 \sin_p(t)\;d\,\cos_p(t),
  \tagEquation\label{eq:area-integral}
\end{align*}
then, \[
  \frac{d\,A(t)}{dt} = \frac{1}{2} \cos_p(t)^p - \frac{1}{2} \sin_p(t)^p - \sin_p(t)^p = \frac{1}{2} (\cos_p(t)^p + \sin_p(t)^p) = \frac{1}{2}.
  \tagEquation\label{eq:dAdt}
\]

With equations \eqref{eq:dLqdt} and \eqref{eq:dAdt}, we can prove the following proportionality relationship between $A(t)$ and $\Lpq(t)$.

\begin{proposition}\label{prop:2D}
  The functions $A(t)/A(\pi_p/2)$ and $\Lpq[p](t)/\Lpq[p](\pi_p/2)$ of relative area and relative $p$-length for $t\in[0,\pi_p/2]$ are equal if and only if $p\in\set{1,2,\oo}$.
  Moreover, these functions are also equal to $\Lpq[2](t)/\Lpq[2](\pi/2)$ if and only if $p\in\set{1,2,\oo}$.
\end{proposition}
\begin{proof}
  Let $q>0$ be the norm used for measuring the relative length.
  Since the relative functions are already equal at 0 and 1, taking values 0 and 1 respectively, it suffices to show that $A'(t)\propto \Lpq'(t)$.
  From Equation~\eqref{eq:dAdt}, we have $A'(t) = 1/2$, so the problem reduces to show that $\Lpq'(t)$ is a positive constant.
  From Equation~\eqref{eq:dLqdt}, it's clear that there are only three cases in which $\Lpq'(t)$ is constant:
  \begin{enumerate}
    \item $p=\oo$ forces $\Lpq'(t)=1$ explicitly for all $q$;
    \item $p=1$ forces $\Lpq'(t)=\norm{(1,1)}_q$ which is constant for all $q$;
    \item if $(p-1)q=p$, then the property $\cos_p(t)^p + \sin_p(t)^p=1$ can be used in Equation~\eqref{eq:dLqdt} to obtain $\Lpq'(t) = (\cos_p(t)^{(p-1)q} + \sin_p(t)^{(p-1)q})^{1/q} = 1$, and this only happens when $p=\frac{q}{q-1}$.
  \end{enumerate}
  With the additional constraint that $q\in\set{2, p}$, there are only five solutions, which are $(p,q)\in\set{(1,1),(1,2), (2,2),(\oo,2), (\oo,\oo)}$.
\end{proof}


\begin{figure}[ht]
  \centering
  \includegraphics[width=0.94\textwidth]{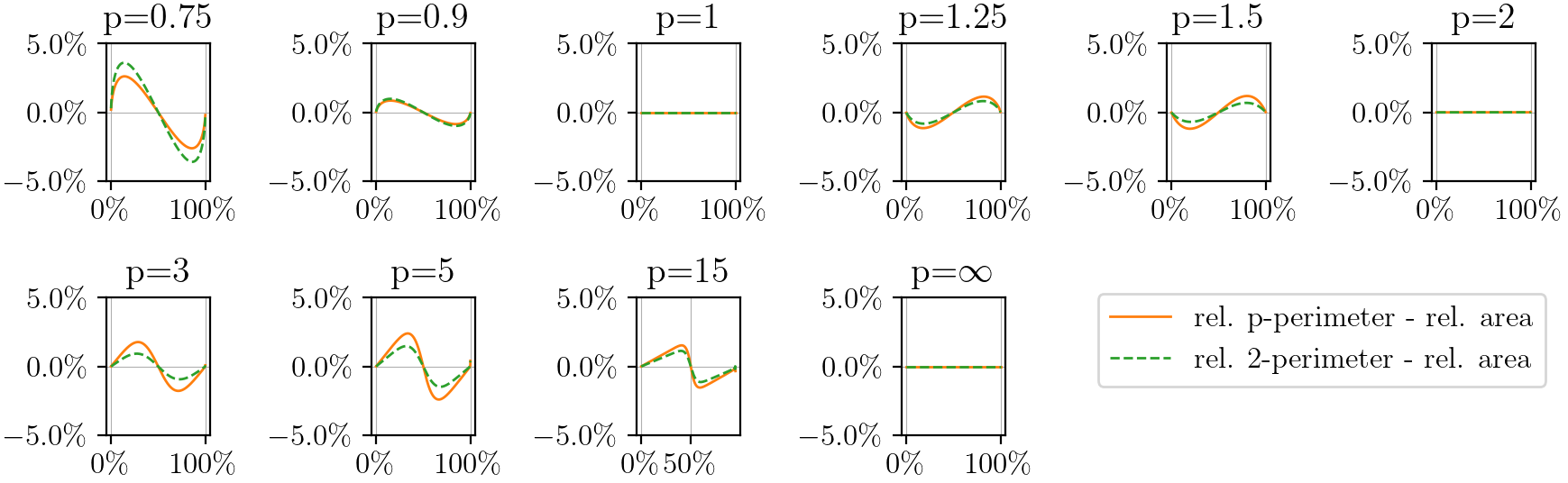}
  \caption{
    Difference between relative perimeters (cumulative divided by total) and relative area varying $t\in[0,\pi_p/2]$.
    The x-axis is the relative area. The two curves are zero only for $p\in\{1,2,\oo\}$.
    }
  \label{fig:relative-diff}
\end{figure}

As a consequence of Proposition~\ref{prop:2D}, it is not the same to sample uniformly from the area of the $p$-circle and project radially to its circumference, as it is to sample uniformly from its circumference directly, unless and only unless $p\in\set{1,2,\oo}$. This is shown in Figure~\ref{fig:relative-diff}.

\subsection{If there is a p-length, why not a p-area?}

While the notion of arc length on a curve depends on the ambient nor, hence allowing for a definition of $q$-length, the same does not hold for area.
The reason lies in how these measures are approximated.

Consider a rectangle with vertices $(0,0)$, $(a,0)$, $(0,b)$, $(a,b)$.
Regardless of the $q$-norm being used to measure the sides, the lengths remain $a$ and $b$, and so the area is independent of $q$, at least for straight (non-tilted) rectangles.
If we assume the following minimal assumptions for a notion of area,
\begin{enumerate}
  \item the area of any figure is non-negative,
  \item the area of the union of disjoint pieces gives the sum of their areas, and
  \item translation invariance,
\end{enumerate}
then, it follows that the area of any (non-fractal or pathologic) 2D figure can be approximated with the sum of areas of many tiny disjoint rectangles that cover it.
For non-pathologic figures, the total area of the rectangles that touch the figure border (overestimating the total area) converges to zero as the resolution increases, hence also the error of the approximation converges to zero.
Since this limiting procedure yields the same result independently of the norm $q$ used to measure the sides of the rectangles, the notion of area for figures in $\bbR^2$ is independent of $q$ and coincides with the classic euclidean measure of area.

For length, however, when segmenting a curve into infinitesimal diagonal segments, the length of each diagonal depends on the spatial norm.
If the curve was segmented into horizontal and vertical segments, the length would be the same for all $q$-norms, but this is known to be an invalid approximation for the length of a curve.

For surfaces in 3D, the approximation consists of infinitesimal inclined planes, and the area of each plane depends on the $q$-norm used to measure its diagonal sides.
The same dependency holds for lengths in 3D, while the volume in 3D is the same for all $q$-norms as it is approximated with straight cubes.
More generally, for $\bbR^n$ endowed with the $q$-norm, the $n$-dimensional measure does not depend on $q$ but all $k$-dimensional measures with $1\leq k<n$ do.

It is important to emphasize that the lengths of the diagonals and the area of inclined planes must be computed without rotating them, because all $p$-norms are sensitive to rotations, except only for the euclidean norm.
In fact, as shown in Theorem 5.5 of~\cite{macarthur2023classifying}, when $p\in[1,\oo)$ but $p\neq 2$, there are only 8 non-translating isometries (transformations that preserve size), namely, the identity, the reflections across the $y$-axis, $x$-axis, the line $y=x$, the line $y=-x$, and the rotations by $\pi/2$ , $\pi$, and $-\pi/2$.
So, defining a rotation in 2D as an isometry that fixes the origin and only the origin, there are only 4 rotations for $p\neq 2$, while for $p=2$ there are infinitely many.
For this reason, at least from an abstract point of view, $p=2$ provides the richest structure and is the most special of all $p$-norms.

\section{Volume vs. surface in nD}

We now extend the 2D area–length analysis to the $n$-dimensional setting.
Let $\Sph^n_p$ denote the $n$-dimensional $p$-ball and let the \emph{unsigned $n$-dimensional $p$-sphere}, denoted $|\Sph^n_p|$, be the subset of $\Sph^n_p$ in which all coordinates are non-negative.
Figure~\ref{fig:3D} shows an instance of $|\Sph^n_p|$.

To parametrize the unsigned $p$-sphere, let us use the following recurrence formula that simplifies the generalization to $n$ dimensions.
Let the last coordinate be given by $z \eqdef \cos_p(t)$ and the remaining by $\sin_p(t) x_i$ for $t\in[0,\pi_p/2]$ and $(x_1, ..., x_{n-1}) \in |\Ball^{n-1}_p|$.

\begin{figure}[ht]
  \hspace{-0.1em}
  \begin{minipage}{\textwidth}
    \centering
    \includegraphics[width=0.6\textwidth]{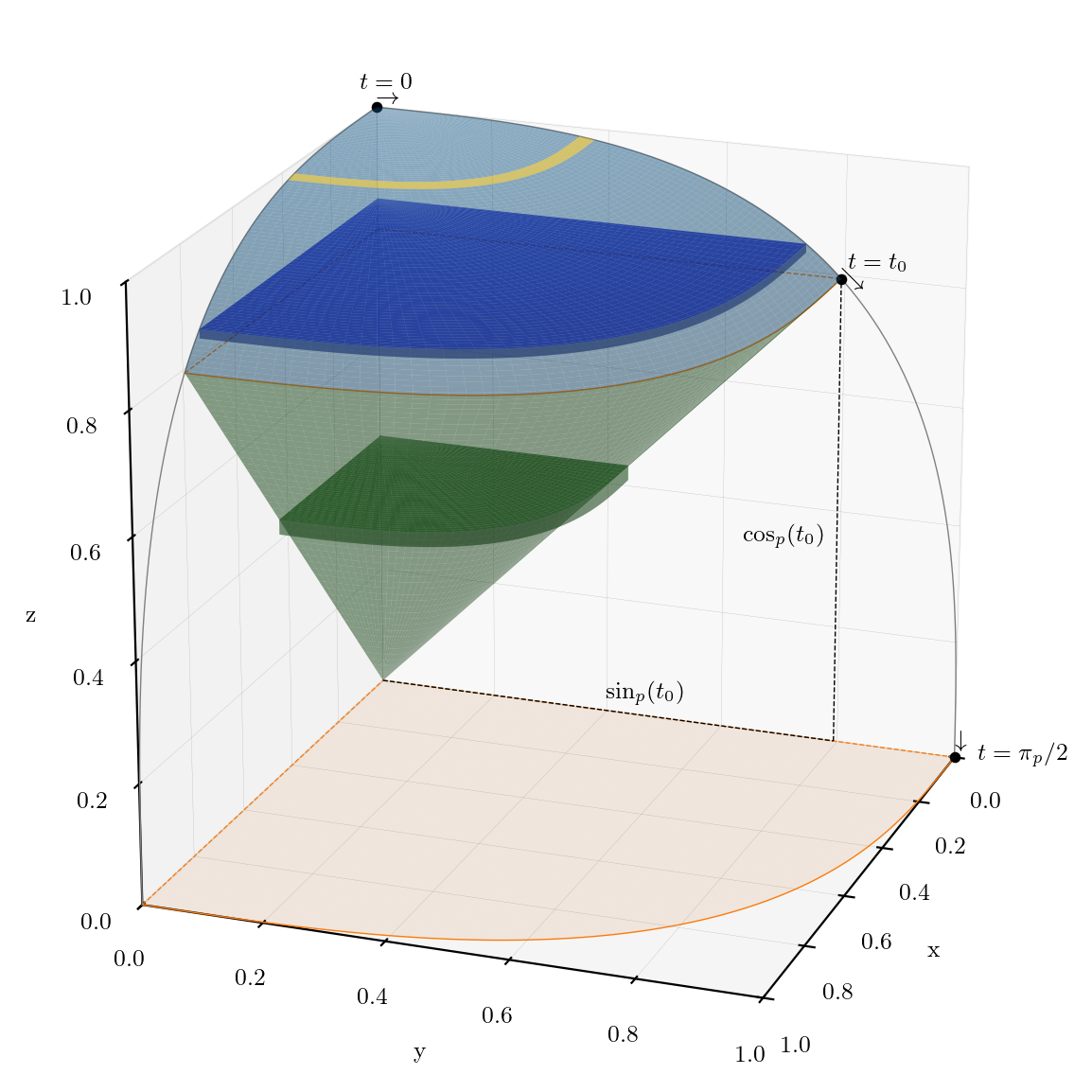}
  \end{minipage}
  \caption{Parametrization of an unsigned $(n{=}3)$-dimensional $(p{=}4)$-ball showing differential volume disks (dark blue, dark green) and surface slices (yellow).}
  \label{fig:3D}
\end{figure}

The objective is to compare the volume $V(t)$ enclosed by the light blue and green shades, with the area $S_q(t)$ of the light blue shade.
The terms \emph{volume}, \emph{area} and \emph{length} denote the $n$, $(n-1)$ and $(n-2)$-dimensional measures.
For this comparison, we shall introduce two constants.
Let $R$ be the area of the orange shade, which is contained in the hyper-plane $z=0$, and let $P_q$ be the length of its solid orange border.

The area $R$ does not depend on $Q$, because it is the mass of an $n-1$-dimensional object that is embedded in $n-1$ dimensions without rotation (by discarding the last dimension).
However, the hyper-length $P_q$ does depend on $q$, analogously to the discussion in the previous section.
In particular, for $n=3$, $P_q$ coincides with $\Lpq(\pi_p/2)$ and $R$ with $A(\pi_p/2)$.

To integrate the volume, we shall consider the dark blue and dark green hyper-disks of differential heights.
The top and bottom hyper-surfaces of these hyper-disks, are scaled versions of the same orange shape.
Notice that when this shape is scaled on every dimension by a factor $c$, the area changes by a factor $c^{n-1}$, because $p$-norms are linear to scalar multiplication and the shape is $(n-1)$-dimensional.
Hence, the area of the blue hyper-disk contained in the plane $z=\cos_p(t)$ for some $t$ is $\sin_p(t)^{n-1} R$, and the area of the dark green hyper-disk at height $z=\cos_p(t)$, is $R \cos_p(t) \lr{\frac{\sin_p(t_0)}{\cos_p(t_0)}}^{n-1}$.
This yields
\begin{align*}
  V(t_0)
  &= \int_0^{t_0} R \sin_p(t)^{n-1}\, d\cos_p(t) + \int_{t_0}^{\pi_p/2} R \left(\frac{\sin_p(t_0) \cos_p(t)}{z(t_0)}\right)^{n-1}\, d\cos_p(t) \\
  &= R \int_0^{t_0} \sin_p(t)^{n-1}\, d\cos_p(t) + R \frac{\sin_p(t_0)^{n-1}}{z(t_0)^{n-1}} \left.\frac{\cos_p(t)^{n}}{n}\right|_{t_0}^{\pi_p/2} \\
  &= R \int_0^{t_0} \sin_p(t)^{n-1}\, d\cos_p(t) + \frac{R}{n} \sin_p(t_0)^{n-1} z(t_0).
\end{align*}
It follows that $V(t)$ is differentiable, with derivative:
\begin{align*}
  \frac{d\,V(t)}{dt}
  &= R \sin_p(t)^{n-1} \frac{d\,\cos_p(t)}{dt} -  \frac{R}{n} \;\frac{d}{dt}\hspace{-0.2em}\left( \sin_p(t)^{n-1} \cos_p(t) \right)\\
  &= R \sin_p(t)^{n+p-2} +  \frac{R (n-1)}{n} \sin_p(t)^{n-2} \cos_p(t)^p  -  \frac{R}{n} \sin_p(t)^{n+p-2}  \\
  &=  \frac{R}{n} \sin_p(t)^{n-2} (n \sin_p(t)^{p} + (n-1) \cos_p(t)^p  - \sin_p(t)^{p})  \\
  &=  \frac{R}{n} \sin_p(t)^{n-2} (n-1).
\end{align*}

Therefore, the hyper-volume can be expressed more succinctly as
\begin{align*}
  V(t_0)
  &= \frac{R\cdot(n-1)}{n} \int_0^{t_0} \sin_p(t)^{n-2} \,dt
  \tagEquation\label{eq:integral V}
\end{align*}

To integrate the hyper-surface, consider the differential surface slice in yellow, which forms an annulus, i.e. the shape of a flat washer.
The width of the annulus is $d\,\Lpq(t)$ (a differential length along the gray curve), while the hyper-length can be computed by scaling $P_q$ by a factor of $\sin_p(t)$, i.e. $\sin_p(t)^{n-2}$ to account for all dimensions.
The area of the annulus is therefore $P_q \sin_p(t)^{n-2} d\,\Lpq(t)$, and that of the hyper-surface is
\[
  S_q(t_0) = P_q \int_0^{t_0} \sin_p(t)^{n-2}\, d\,\Lpq(t)
  \tagEquation\label{eq:integral Sq}
\]

Putting together Equations~\eqref{eq:integral V} and~\eqref{eq:integral Sq}, we can prove the following proposition that generalizes Proposition~\ref{prop:2D}.

\begin{proposition}\label{prop:nD}
  Let $n\geq 2$. The relative hyper-volume $V(t)/V(\pi_p/2)$ and hyper-$p$-area $S_p(t)/S_p(\pi_p/2)$ coincide for all $t\in[0,\pi_p/2]$ if and only if $p\in\set{1,2,\oo}$.
  Moreover, they also coincide with $S_2(t)/S_2(\pi/2)$ for every such $t$ if and only if $p\in\set{1,2,\oo}$.
\end{proposition}
\begin{proof}
  As in the 2-dimensional case (Proposition~\ref{prop:2D}), the relative functions are already equal at 0 and 1, so the problem is equivalent to showing that $V'(t)\propto S_p'(t)$, which according to formulas~\eqref{eq:integral V} and~\eqref{eq:integral Sq} is equivalent to showing that $\Lpq'(t)$ is constant.
  From the proof of Proposition~\ref{prop:2D}, it follows that this occurs with $q\in\set{p,2}$ if and only if $p\in\set{1,2,\oo}$.
\end{proof}

As a consequence of Proposition~\ref{prop:nD}, it can be concluded for all $n$ that it is not the same to sample uniformly from the hyper-volume of a $p$-ball and project radially to its hyper-surface, as it is to sample uniformly from its hyper-surface directly, unless, and only unless, $p\in\set{1,2,\oo}$.

The proof of Proposition~\ref{prop:nD} does not require to determine the constants $R$ and $P_q$ in Equations~\eqref{eq:integral V} and~\eqref{eq:integral Sq}.
For completeness, we show how to compute these next.
Let $V_{n} \eqdef 2^n V(\pi_p/2)$ denote the total volume of the $n$-dimensional $p$-ball, and $S_{n,q} \eqdef 2^n S_q(\pi_p/2)$ its total area.
Under this notation, we have $R = V_{n-1} / 2^{n-1}$ and $P_q = S_{n-1,q} / 2^{n-1}$.
Using Equations~\eqref{eq:integral V} and~\eqref{eq:integral Sq}, we obtain recursive formulas for $V_{n}$ and $S_{n,q}$, whose simplification leads to
\begin{align*}
  V_{n} &= \frac{2^n}{n} \prod_{k=0}^{n-2} \int_0^{\pi_p/2} \sin_p(t)^{k} \,dt, \\
  S_{n,q} &= 2^n \prod_{k=0}^{n-2} \int_0^{\pi_p/2} \sin_p(t)^{k} \,d\Lpq(t).\\ 
\end{align*}

\section{Sampling algorithms}

In this section we develop and discuss computational methods to sample uniformly from the unitary $p$-ball in $n$ dimensions, both volumetrically inside the ball and superficially on the frontier.
An implementation of these algorithms in Python can be found in the appendix.

\subsection{Algorithms based on squigonometry}

Based on the results of the previous section, particularly on the recursive construction of the unsigned $p$-sphere as well as Equations~\eqref{eq:integral V} and \eqref{eq:integral Sq}, we can derive the following algorithms by sampling coordinate by coordinate recursively.

\makeatletter
\newcommand{\staticlabel}[1]{#1\protected@edef\@currentlabel{#1}}
\makeatother

\textbf{Algorithms \staticlabel{V1}\label{alg:volume-coord} and \staticlabel{S1}\label{alg:surface-coord}}. Generators of random samples from the $n$-dimensional unitary $p$-ball.
Samples from Algorithm~\ref{alg:volume-coord} are uniformly distributed in the hyper-volume (in any $q$-norm).
Samples from Algorithm~\ref{alg:surface-coord} are uniformly distributed on the hyper-surface in the specified $q$-norm.
\begin{enumerate}
  \item Let $X_1 = 1$.
  \item For $k=2..n$:
  \begin{enumerate}
    \item Sample $T_k \in [0, \pi_p/2]$ with law $f_{T_k}(t)$, defined as:\\*$
      \begin{cases}
        \text{For Algorithm~\ref{alg:volume-coord} (volume), }&f_{T_k}(t)\propto \sin_p(t)^{k-2}.\\
        \text{For Algorithm~\ref{alg:surface-coord} (surface), }&f_{T_k}(t)\propto \sin_p(t)^{k-2} \frac{d\,\Lpq(t)}{dt}.
      \end{cases}
      $
    \item Let $X_k = \cos_p(T_k)$.
    \item Update $X_i \leftarrow \sin_p(T_k) X_i$ for all $i=1,2,...,k-1$.
  \end{enumerate}
  \item Sample signs $S_1, ..., S_n\sim  \Uni\set{-1, 1}$.
  \item Return $(S_i\, X_i)_{i=1}^n$.
\end{enumerate}

The implementation of these algorithms has the challenges of computing $\cos_p(t)$ and $\sin_p(t)$, and above all, sampling $T_k$ from $f_{T_k}(t)$.
By using a fine grid over a single unsigned $p$-sphere in 2 dimensions, it is possible to accurately approximate the inverse $p$-cosine and $p$-sine functions as well as the cumulative distribution function (CDF) of $f_{T_k}(t)$ for each value of $k$.
These functions can then be inverted using interpolation to obtain $\cos_p(t)$ and $\sin_p(t)$ and to sample $T_k$.
For a complete implementation, refer to the code in the appendix.

\subsection{The p-normal distributions}

The gaussian distribution has the distinctive property that if $X$ and $Y$ are independent and gaussian, then the joint density function $f_{X,Y}(x,y)$ of the vector $(X, Y)$ is isotropic (i.e. radially symmetric), as it can be written as a function of the euclidean norm $f_{X,Y}(x,y) = f(x) f(y) = g(\norm{(x,y)}_2)$.
Moreover, the joint of independent distributions is isotropic if and only if it is a scaled version of the gaussian.

The argument is the following\footnote{A video visualizing and explaining the proof can be found in~\cite{youtube3blue1brown}.}.
For $y=0$ we have $f(x) f(0) = g(|x|)$, so letting $a=\sqrt{|x|}$, $b=\sqrt{|y|}$, and $h(r) \eqdef g(\sqrt{r}) / f(0)^2$, the radial symmetry equation becomes \[
  h(a^2) h(b^2)
  = \frac{g(a)}{f(0)^2} \frac{g(b)}{f(0)^2}
  = \frac{f(\pm a)}{f(0)} \frac{f(\pm b)}{f(0)}
  = \frac{g(\norm{(\pm a,\pm b)}_2)}{f(0)^2}
  = h(a^2 + b^2).
\]
This functional equation is known as the Exponential Cauchy's Functional Equation~\cite{kannappan2009functional-2-2}, and the only non-degenerate\footnote{Pathological solutions exist~\cite{book:28165471} by means of Hamel bases. See~\cite{chen2016monsters} for a quick and clear explanation for the “monstrous” solutions to Cauchy's additive functional equation, which are analogous to the exponential via a logarithmic transformation.} solution is known to be the exponential function~\cite{gaughan2009introduction-4-39,stromberg1965real-ex18-46} $h(a)=e^{c\,x}$.
Therefore, $f$ must be given by $f(x) = f(0)\, e^{-c |x|^2}$ for some constant $c$.
For $f$ to be a density, we must have $c>0$ and $f(0)=c/\sqrt{2\pi}$, where the mysterious appearance of the constant $\pi$ is due precisely to the radial symmetry of the gaussian, a beautiful connection and celebrated fact for which at least 11 different proofs have been found~\cite{conrad2016gaussian}.

With a similar argument, it can be shown that a distribution is radially symmetric in the $p$-norm for $p<\oo$ if and only if it is given by $f(x) = f(0)\, e^{-c |x|^p}$ for some constant $c>0$.
For $p=\oo$, the argument is slightly different.
If $f(y)>0$ for some $y>0$, then for all $x\in[0,y]$, we have $f(\pm x) f(\pm y) = g(\norm{(x,y)}_\oo) = g(y)$, so $f$ is constant on $[-y,y]$.
Therefore, $f$ must be given by $f(x) = f(0) \one{|x|<c}$ for some $c > 0$.

Both for $p<\oo$ and $p=\oo$, the scale of $f$ can be chosen arbitrarily in principle, but for the purposes of this paper, we define the \emph{$p$-normal distribution} $f_p$ as
\begin{align*}
  f_p(x) &\eqdef c_p \exp(-|x|^p/p)\quad;\quad c_p=\frac{1}{2\,\sqrt[p]{p}\,\;\G(1+1/p)},\\
  f_\oo(x) &\eqdef \lim_{p\to\infty} f_p(x) = c_\oo \one{|x|<1}\quad;\quad c_\oo=\lim_{p\to\infty} c_p=\frac{1}{2},
  \tagEquation\label{eq:fp}
\end{align*}
where the default scale guarantees that $E_{X\sim f_p}(\norm{X}_p)=1$.
For different scales ($b\neq 1$), the distribution of $Y=b X$ is given by $f_{p,b}(y) = b f_p(y/b)$.

This distribution has already been investigated with different names in the literature, including the $\theta$-normal distribution~\cite{goodman1973multivariate} ($\propto\exp(-|x|^p)$), the General Gaussian distribution~\cite{livadiotis2012expectation,livadiotis2014chi}, and without a particular name in~\cite{song1997lp-uniform}.
Except for the lattermost scale factor (as a function of $p$) is different.
The reason for this is that these studies are less focused on geometry and more focused on the statistics of the distribution, e.g. the so-called $L^p$-variance.
In this regard, our paper can be regarded as a bridge between these works that focus in statistics and that of~\cite{wood2020squigonometry-chaps3-7}, mostly focused on geometry and calculus.

Moreover, Equation~\eqref{eq:fp} also corresponds to a signed power gamma distribution, a particular case of the signed generalized gamma distribution.
The generalized gamma has been discussed at least since 1924~\cite{amoroso1925ricerche-pp124} with the purpose of studying income curves, and studied in depth since then with different parametrizations and names~\cite{stacy1962generalization,johnson1995continuous,kleiber2003statistical,khodabina2010some-pp11}, e.g. generalized Weibull.

Let $X_1,...,X_n$ be i.i.d. random variables with law $f_p$.
From Equation~\eqref{eq:fp}, the law of the vector $X=(X_1,...,X_n)$ is given by
\begin{align*}
  f_p\nD(x) &\eqdef f_p(x_1)\cdots f_p(x_n) =
  \begin{cases}
    c_p^n \exp(-\norm{x}_p^p/p) & \text{ if }p<\oo,\\
    c_\oo^n\; \one{\norm{x}_\oo < 1} & \text{ if }p=\oo,
  \end{cases} \\
  & = c_p^{n-1}\;f_p(\norm{x}_p).
  \tagEquation\label{eq:fpnD}
\end{align*}
with $c_p$ as indicated in Equation~\eqref{eq:fp}.
In both cases, the joint density $f_p\nD$ is $p$-isotropic, i.e. radially symmetric in the $p$-norm, as it can be written as a function of the the $p$-radius $R_{n,p}=\norm{X}_p$.

The density of $R_{n,p}$ is given by the integration of Equation~\eqref{eq:fpnD} over $p$-spheres.
A $p$-sphere of radius $r$ has a mass proportional to $r^{n-1}$, and the density of $R_{n,p}$ is therefore given by
\begin{align*}
  f_{R_{n,p}}(r) \propto r^{n-1} \exp(-(n/p)\,r^p)\quad ;\quad f_{R_{n,\oo}}(r) \propto r^{n-1} \one{r < 1}.
  \tagEquation\label{eq:fpr}
\end{align*}
By considering a scaled radius $R^\star_{n,p}\eqdef R_{n,p} / \sqrt[p]{n}$ and with the notation $\sqrt[\oo]{n}\eqdef \lim_{p\to\oo}\sqrt[p]{n} = 1$, Equation~\eqref{eq:fp} simplifies Equation~\eqref{eq:fpr} into $f_{R^\star_{n,p}}(r) \propto r^{n-1} f_p(r)$.

\subsection{Algorithms based on the p-normal distributions}

Let $X$ be the random vector $X=(X_1,...,X_n)$ composed of i.i.d. random coordinates with law $f_p$.
Since $X$ is $p$-isotropic, for any fixed radius $r$ and as $\Delta r \to 0$, the conditional distribution of $X$ given $\norm{X}_p\in(r, r+\Delta r)$ converges to a uniform distribution in its domain---the volume between two $p$-spheres with radius difference of $\Delta r$.

The conditional distribution of the vector $\bar X \eqdef X/r$ given $\norm{X}_p\in(r, r+\Delta r)$ also converges as $\Delta r \to 0$ to a uniform distribution in the volume between the $p$-spheres of radius $1$ and $1+\frac{\Delta r}{r}$.
Notice that the same conclusion is valid for any $X$ that is $p$-isotropic, therefore the distribution of $\bar X$ is invariant.

In particular, if $V$ is a random vector uniformly distributed in the unitary $p$-ball, it is $p$-isotropic, hence $V/\norm{V}_p$ is distributed like $X$.
Furthermore, the distribution of $R\eqdef \norm{V}_p$ can be obtained from the same observation that produced Equation~\eqref{eq:fpr}, namely, $f_{R}(r) \propto r^{n-1}$, i.e., a standard power distribution $R\sim\Beta(1, n)$, or $R=U^{1/n}$ where $U\sim\Uni[0,1]$.
These two facts can be combined to reconstruct the variable $V$ as the product $R \bar X$, as presented in Algorithm~\ref{alg:volume-normal}.
However, the same argument does not hold for the hyper-surface (Algorithm~\ref{alg:surface-normal}), unless $p\in\set{1,2,\oo}$.
This is discussed in the next section.

Algorithms~\ref{alg:volume-normal} and~\ref{alg:surface-normal} are not novel.
They have already been presented in different formats in~\cite{cambanis1981theory,song1997lp-uniform}.
In particular, \cite{song1997lp-uniform} derives the density function of Algorithm~\ref{alg:surface-normal} as well as its marginal distribution.
Nevertheless, the distribution of Algorithm~\ref{alg:surface-normal} is said to be “uniformly distributed on the surface of the (Lp-norm) unit sphere”, by definition or intuition, which is a claim contested here and discussed in the next section.

\textbf{Algorithms \staticlabel{V2}\label{alg:volume-normal} and \staticlabel{S2}\label{alg:surface-normal}}.
Simpler alternatives to Algorithms~\ref{alg:volume-coord} and \ref{alg:surface-coord}, respectively.
Algorithm~\ref{alg:volume-normal} is equivalent to Algorithm~\ref{alg:volume-coord} for all $p$, while Algorithm~\ref{alg:surface-normal} is equivalent to Algorithm~\ref{alg:surface-coord} only for $p\in\set{1,2,\oo}$. For other $p$, it fails to sample uniformly from the hyper-surface.
\begin{enumerate}
  \item Sample $Z_1, ..., Z_n \sim \Gamma(n/p)$.
  \item Let $X_i \leftarrow Z_i^{1/p}$ for all $i=1,2,...,n$.
  \item Let $r \leftarrow \norm{(X_1, ..., X_n)}_p$ and $\bar X_i \leftarrow X_i/r$ for all $i=1,2,...,n$.
  \item $
    \begin{cases}
      \text{For Algorithm~\ref{alg:volume-normal} (volume),}&\text{ sample }R \sim \Beta(1, n).\\
      \text{For Algorithm~\ref{alg:surface-normal} (surface),}&\text{ let }R \gets 1.
    \end{cases}
    $
  \item Sample signs $S_1, ..., S_n\sim  \Uni\set{-1, 1}$.
    \item Return $(S_i\, \bar X_i\, R)_{i=1}^n$.
\end{enumerate}

The vector $Z=(Z_1, Z_2, ..., Z_n)$ in Algorithm~\ref{alg:volume-normal} is $p$-isotropic.
In particular, for $p\in\set{1,2,\oo}$, this type of distribution is very used in physics, (mostly $p\in\set{2,\oo}$ and $n\in\set{2,3}$), as well as in other areas of science and engineering.
They are used, for instance, in the initialization and noise production for attacks to neural networks (mostly $p=2$, $n \gg 1$, e.g. Algorithm 1 in~\cite{picot2022halfspace}), and in noise-injection for privacy applications including the laplace mechanism ($p=1$, $n=1$), metric-privacy ($p=1$, $n\in\set{2,3}$)~\cite{andres2013geo} and differentially private stochastic gradient descent ($p=1$, $n\gg 1$) in neural network training.

\section{Discussion}
\label{sec:discussion}

The previous sections introduced two different families of sampling algorithms for $p$-balls and $p$-spheres.
While the volume sampling algorithms are equivalent across methods, the surface sampling algorithms differ in a subtle but crucial way: radial projection preserves uniformity only when $p\in\set{1,2,\oo}$.

In this section, we clarify the conceptual reason behind this phenomenon by relating it to the Borel–Kolmogorov paradox and the role of conditioning on null events.
We then present experimental results supporting the theoretical conclusions.

\subsection{The Borel-Kolmogorov paradox}

The Borel-Kolmogorov paradox is a well-known counterintuitive property in probability theory that arises when trying to define a conditional distribution restricted to a set of measure zero, like a surface on a volume, or a curve on a surface.
This paradox and its resolution serve to illustrate that although Algorithm~\ref{alg:surface-normal} is $p$-isotropic, it does not sample uniformly from the hyper-surface of the $p$-sphere for $p\not\in\set{1,2,\oo}$.

The classic instance of the paradox considers a sphere equipped with the uniform distribution as well as two identical curves\em a meridian line and half the equator\em, and finds different expressions for their conditional distributions.
More precisely, if the domain is restricted to any thin lune bounded by consecutive longitudes in a high resolution grid, then the distribution of area is uniform over the lune while the distribution of latitudes is more concentrated around the equator, where the lune is thicker.
As the resolution increases, the lune approaches a meridian path from pole to pole, and the distribution of latitudes converges to $f_\theta(\theta) \propto \sin \theta$, which is maximal at the equator ($\theta=\pi/2$) and minimal at the poles ($\theta \in\set{0,\pi}$).
Meanwhile, the distribution of longitudes for a fixed latitude converges to a uniform distribution in $[0,2\pi]$, so the distribution of half the equator is uniform.
This difference between the distribution of latitudes on a fixed meridian line and that of longitudes over a half of the equatorial line contradicts conventional wisdom, because the two curves are identical and the spherical shape and the uniform distribution on its surface are invariant under rotations.
There is no reason to believe that they should be different.

There are two solutions to the paradox.
The classic solution is that the conditional distribution given the a fixed event depends on the limiting procedure used to approach the event when the event has probability zero (a null event).
This implies that conditional distributions on null events should always have a subscript indicating the limiting procedure, or a context from which it is obvious. 
According to~\cite{jaynes2003probability470}, “...it is difficult to get people to see that the term ‘great circle’ is ambiguous until we specify what limiting operation is to produce it. The intuitive symmetry argument presupposes unconsciously the equatorial limit; yet one eating slices of an orange might presuppose the other.”
He argued that in absence of information about a limiting procedure, it does not make sense to speak about conditional distributions given null events.
But recently, some researchers have proposed an alternative solution to the paradox: in absence of a reference limiting procedure the Hausdorff measure should be understood~\cite{bungert2020lion}, so the answer to the paradox is the uniform distribution, unless a particular limiting procedure is indicated.
In this paper, we share and promote the view that the canonical $k$-dimensional measure in $n$-dimensions is the $k$-dimensional Hausdorff measure, defined and explained in full detail in~\cite{edgar2008measureHausdorff}.
This measure is so natural that we have been referring to lengths of curves of surfaces in $\bbR^n$ without questioning whether the curves have been approached as very thin worms or very thin rosaries.

The results of this paper are closely related to the Borel-Kolmogorov paradox.
The distributions of $\bar X$ and $X$ in Algorithm~\ref{alg:surface-normal} are related in that $X$ is $p$-isotropic and its density can be written as a product\[
  f_X(x) = f_{\bar X}\lr*{\frac{x}{\norm{x}_p}} f_R(\norm{x}_p), \text{ where } f_{\bar X}(\,\cdot\,) \text{ is a constant function.}
  \tagEquation\label{eq:disintegration}
\]
But this does not imply that the distribution of $\bar X$ is uniform over the $p$-sphere \emph{in the canonical sense}.
While Algorithm~\ref{alg:surface-normal} ensures volumetric uniformity between two $p$-spheres of infinitesimally close $p$-radii, Algorithm~\ref{alg:surface-coord} guarantees superficial uniformity in the canonical sense, and, as proven in Proposition~\ref{prop:nD}, these distributions are different unless $p\in\set{1,2,\oo}$.

\subsection{Experimental validation}
\label{sec:empirical}

We corroborate the theoretical results by an experiment comparing Algorithms~\ref{alg:surface-coord} and~\ref{alg:surface-normal}.

\begin{figure}[ht]
  \begin{minipage}{\textwidth}
    \centering
    \includegraphics[width=0.6\textwidth]{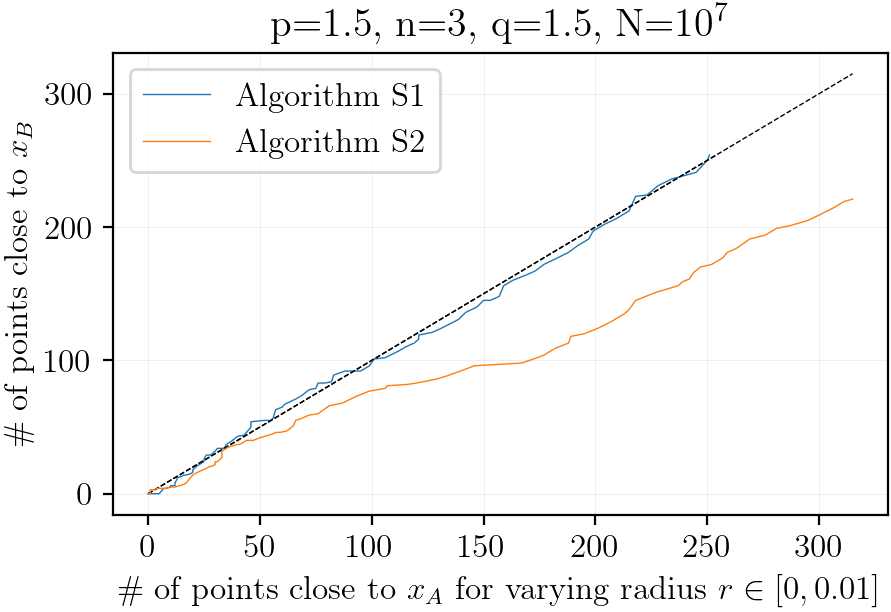}
  \end{minipage}
  \caption{Comparison of neighbor counts around two reference points. Uniformity holds only for Algorithm~\ref{alg:surface-coord}.}
  \label{fig:empirical}
\end{figure}

For a fixed value of $p\notin \set{1,2,\oo}$, we generate a large number $N\gg 1$ of points on the 
$p$-sphere using each algorithm, and select two reference points: $x_A = (0,0,1)$ and $x_B = (\sin_p(\pi_p/5), \cos_p(\pi_p/5)\sin_p(\pi_p/5), \cos_p^2(\pi_p/5))$.
We then count the number of sampled points falling within a small fixed $q$-distance $r\ll 1$ from each reference point.
Under a uniform distribution, these counts should be proportional to the sphere area that is covered by a small $q$-ball and thus approximately equal for both reference points.

Figure~\ref{fig:empirical} shows the counts are aligned for Algorithm~\ref{alg:surface-coord} but not for Algorithm~\ref{alg:surface-normal}.

\section{Conclusion}

This paper has shown that among all $p$-norms in $n$-dimensional space ($n\geq 2$), only $p\in\set{1,2,\oo}$ exhibit the property that when the $X\in\Ball_p^n$ follows a volumetrically uniform distribution, then the radial projection $X\mapsto X/\norm{X}_p $ yields a uniform distribution on the surface $\Sph_p^n$.

Four algorithms for sampling uniformly from the volume and surface of the $p$-ball have been derived and implemented in Python in the Appendix.
The first two are novel and based on recursive squigonometric parametrizations, while the last two are borrowed from the statistical works of~\cite{cambanis1981theory,song1997lp-uniform}, and are based on $p$-normal distribution.
The volume algorithms are distributionally equivalent, but the surface algorithms differ unless $p\in\set{1,2,\oo}$.

This asymmetry is explained through the Borel–Kolmogorov paradox: conditioning on a lower-dimensional subset without controlling the limiting procedure leads to non-uniform distributions.
Finally, experimental results corroborate the theoretical predictions, illustrating the failure of naive radial sampling for $p\notin\set{1,2,\oo}$.

\bibliographystyle{plain} 
\bibliography{main}


\newpage
\appendix

\section*{Appendix}
\subsection*{Python implementation of the algorithms}

For Algorithms~\ref{alg:volume-normal} and~\ref{alg:surface-normal}, sampling from $f_p$ can be done using a generalized gamma distribution.

For Algorithms~\ref{alg:volume-coord} and~\ref{alg:surface-coord}, there are more difficulties.
In order to sample $T_k$ from $f_{T_k}$, one can use an approximation by literally creating a grid with a large number of points that discretizes the unsigned $p$-circumference in 2D, spaced conveniently to improve precision at the critical points $t=0$, $t=\pi/4$ and $t=\pi/2$.
This produces cumulative density functions whose inverse can then interpolate at random values in $(0,1)$.

{
\fontsize{5pt}{5pt}
\begin{lstlisting}{python}
import numpy as np
import scipy.special
from scipy.stats import gengamma
from functools import lru_cache  # Python's standard library

def sample_generators(p: float, n: int, q=None, grid_precision: int = 100000):
    if q is None:
        q = p

    def gen_volume(N):
        if p == np.inf:
            x = np.random.random(size=(N, n))
        else:
            x = gengamma(a=1 / p, c=p).rvs(size=(N, n))
        r = np.linalg.norm(x, ord=p, axis=-1, keepdims=True)
        u = np.random.random(size=(N, 1))
        x = x * np.random.choice([-1, 1], size=x.shape)
        return x / r * u ** (1 / n)

    def gen_surface(N):
        if p in (1, 2, np.inf):  # Use volume projection (faster)
            x = gen_volume(N)
            return x / np.linalg.norm(x, ord=p, axis=-1, keepdims=True)

        x = np.empty((N, n))
        x[:, 0] = 1
        x_grid, y_grid = grid_info["x"], grid_info["y"]
        for k in range(2, n + 1):
            u = np.random.random(N)
            x_k = np.interp(u, xp=t_k_curve(k), fp=x_grid)
            y_k = np.interp(u, xp=t_k_curve(k), fp=y_grid)
            x[:, k - 1] = x_k
            x[:, : k - 1] *= y_k[:, None]
        x *= np.random.choice([-1, 1], size=x.shape)
        return x

    # Auxiliary functions for surface sampling

    @lru_cache(maxsize=n)  # cache to avoid recomputations
    def t_k_curve(k):
        y, dLq = grid_info["y"], grid_info["dLq"]
        cum_t_k = np.cumsum(y ** (k - 2) * dLq)
        return cum_t_k / cum_t_k[-1]

    def p_circumference_grid():
        # (x,y) are the points in the p-circumference. w is y^p.
        # w_eps = min(0.05, max(1 / grid_precision ** (q * p), 1e-16))
        geom = np.geomspace(1e-16, 0.1, num=int(0.05 * grid_precision))
        lin = np.linspace(0.1, 0.4, num=int(0.4 * grid_precision))
        w = np.array([0, *geom, *lin, *(0.5 - geom[::-1])])
        w = np.unique(np.sort([0, *w, 0.5, *(1 - w[::-1]), 1]))
        y = w ** (1 / p) if p < np.inf else np.minimum(2 * w, 1.0)
        x = (1 - w) ** (1 / p) if p < np.inf else np.minimum(2 * (1 - w), 1.0)
        dY = np.diff(y, prepend=0)
        dX = np.diff(x, prepend=1)
        dA = (0.5 * dX * y + 0.5 * x * dY) - y * dX
        t = 2 * np.cumsum(dA)
        dLq = np.linalg.norm([np.abs(dX), np.abs(dY)], ord=q, axis=0)
        Lq = np.cumsum(dLq)
        pi_p, piL_pq = t[-1] * 2, Lq[-1] * 2
        return dict(t=t, x=x, y=y, Lq=Lq, dLq=dLq), pi_p, piL_pq

    grid_info, pi_p, piL_pq = p_circumference_grid()

    return gen_volume, gen_surface, pi_p, piL_pq


# Usage example
gen_volume, gen_surface, *_ = sample_generators(p=1.5, n=3)
xyz_v = gen_volume(10000)
xyz_s = gen_surface(10000)

import plotly.graph_objects as go
_kw = dict(mode="markers", marker=dict(size=2))
kw = lambda xyz: dict(x=xyz[:, 0], y=xyz[:, 1], z=xyz[:, 2], **_kw)
fig = go.Figure()
fig.add_trace(go.Scatter3d(**kw(xyz_v)))
fig.add_trace(go.Scatter3d(**kw(np.abs(xyz_s))))
fig.show()

# Squigonometric functions

def squigonometry(p):
    pi_p = 4 if p == np.inf else scipy.special.beta(1 / p, 1 / p) * 2 / p
    acos_p_ = lambda t: scipy.special.betainc(1 / p, 1 / p, t**p) * 2 / pi_p / p
    sin_p__ = lambda t: scipy.stats.beta(1 / p, 1 / p).ppf(t) ** (1 / p)
    sin_p_ = lambda t: np.sign((t + 2) % 4 - 2) * sin_p__(np.abs((t + 1) % 2 - 1))
    cos_p_ = lambda t: sin_p_(t + 1)
    sin_p_oo = lambda t: np.sign((t + 2) % 4 - 2) * np.minimum(
        1, 2 * np.abs((t + 1) % 2 - 1)
    )
    cos_p_oo = lambda t: sin_p_oo(t + 1)
    cos_p = lambda t: (cos_p_(t / (pi_p / 2)) if p < np.inf else cos_p_oo(t / 2))
    sin_p = lambda t: (sin_p_(t / (pi_p / 2)) if p < np.inf else sin_p_oo(t / 2))
    return pi_p, cos_p, sin_p


# Functions for uniformity test:
def curve(samples, ref):
    dists = np.linalg.norm(samples - ref, ord=q, axis=-1)
    # cut to analyze local behavior only (manifold)
    dists = dists[dists < delta]
    x = np.linspace(0, delta, bins)
    y = np.bincount(np.digitize(dists, x), minlength=len(x))
    return np.cumsum(y)

p = 1.5
q = p
n = 3
gen_volume, gen_surface, *_, V, S = sample_generators(p=p, q=q, n=n)

delta = 0.01
bins = 100
N = 10**7
np.random.seed(0)  # Exact reproducibility
points_S1 = gen_surface(N)
points_S2 = gen_volume(N)
points_S2 /= np.linalg.norm(points_S2, ord=p, axis=-1, keepdims=True)


# Curves for comparison
pi_p, cos_p, sin_p = squigonometry(p)
x_A = np.array([0, 0, 1])
x_B = np.array(
    [sin_p(pi_p / 5), cos_p(pi_p / 5) * sin_p(pi_p / 5), cos_p(pi_p / 5) ** 2]
)
curves_S1 = [curve(points_S1, ref) for ref in [x_A, x_B]]
curves_S2 = [curve(points_S2, ref) for ref in [x_A, x_B]]

import matplotlib.pyplot as plt

plt.plot(curves_S1[0], curves_S1[1], label="S1")
plt.plot(curves_S2[0], curves_S2[1], label="S2")
plt.plot([0, 250], [0, 250], "k--", lw=0.5)
plt.legend()
plt.show()
\end{lstlisting}
}

\includegraphics[width=0.6\textwidth]{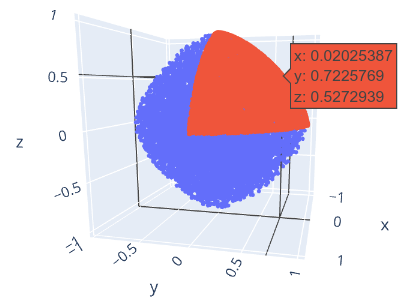}

\end{document}